\documentclass[12pt]{article}

\usepackage[top=50mm, bottom=50mm, left=50mm, right=50mm]{geometry}
\usepackage{lineno}
\usepackage{amssymb}
\usepackage{amsmath}
\usepackage{amsthm}
\usepackage{bbold}
\usepackage{epsfig}
\usepackage{faktor}
\usepackage{graphicx}
\usepackage{graphics}
\usepackage{pifont}
\usepackage{float}
\usepackage{subfigure}
\usepackage{multirow}
\usepackage{color}
\usepackage{lineno}
\usepackage{fullpage}
\usepackage[normalem]{ulem} 
\usepackage{makeidx}
\usepackage{xspace}
\usepackage{wrapfig}
\usepackage{hyperref}
\usepackage{cleveref}
\usepackage{pst-node}
\usepackage{tikz-cd} 
\usepackage{mathtools}

\makeindex

\newtheorem{Definition}{Definition}[section]
\newtheorem{theorem}[Definition]{Theorem}
\newtheorem{Corollary}[Definition]{Corollary}

\newtheorem{Lemma}[Definition]{Lemma}

\newtheorem{Notation}[Definition]{Notation}

\newtheorem{Proposition}[Definition]{Proposition}

\newtheorem{Remark}[Definition]{Remark}
\newtheorem*{theorem*}{Theorem}
\newtheorem{Question}[Definition]{Question}
\newtheorem{Conjecture}[Definition]{Conjecture}

\def\gc{{\gamma}}
\def\dgc{\dot{\gc}}

\newcommand {\mm}[1] {\ifmmode{#1}\else{\mbox{\(#1\)}}\fi}

\newcommand{\C}{\mathbb{C}} 
\newcommand{\cp}{\mathbb{C}\mathrm{P}} 
\newcommand{\rp}{\mathbb{R}\mathrm{P}} 
\newcommand{\R}{\mathbb{R}} 
\newcommand{\Z}{\mathbb{Z}} 
\newcommand{\N}{\mathbb{N}} 
 
\newcommand{\hz}{\mathrm{HZ}}

\newcommand{\dd}{\mathrm{d}}

\begin{document}

\title{Magnetic billiards and the Hofer--Zehnder capacity of disk tangent bundles of lens spaces}
 
\author{{Johanna Bimmermann\footnote{Fakultät für Mathematik, Ruhr-Universität Bochum, Universitätsstraße 150, 44801 Bochum, Germany, johanna.bimmermann@rub.de}\ \& Levin Maier\footnote{Mathematisches Institut, Ruprecht-Karls-Universität Heidelberg, Im Neuenheimer Feld 205, 69120 Heidelberg, Germany, lmaier@mathi.uni-heidelberg.de} }}

\maketitle

\begin{abstract}
\noindent
We compute the Hofer--Zehnder capacity of disk tangent bundles of certain lens spaces with respect to the round metric. Interestingly, we find that the Hofer--Zehnder capacity does not see the covering, i.e.\ the capacity of the disk tangent bundle of the lens space coincides with the capacity of the disk tangent bundle of the 3-sphere covering it. In particular, this gives a first example, where Gromov width and Hofer--Zehnder capacity of a disk tangent bundle disagree. We use techniques that include magnetic billiards for the lower bound and Gromov--Witten invariants for the upper bound.
\end{abstract}

\section{Introduction}
Since Gromov's famous non-squeezing theorem, symplectic embedding problems lie at the heart of symplectic geometry. To obtain obstructions for symplectic embeddings, symplectic capacities were introduced as numerical invariants of symplectic manifolds. A normalized symplectic capacity is a function $c$ that assigns a number $c(M,\omega)\in[0,\infty]$ to a symplectic manifold $(M^{2n},\omega)$ of fixed dimension $2n$, such that the following axioms hold \begin{enumerate}
    \item \textit{(Monotonicity)} if $(M,\omega)$ symplectically embeds into $(\Tilde{M}, \Tilde{\omega})$ then $c\left(M, \omega \right)\leq c\left(\Tilde{M}, \Tilde{\omega} \right)$,
    \item \textit{(Conformality)} for all $r>0$, we have $c\left(M,r\omega \right)= r c\left(M,\omega \right)$,
    \item \textit{(Normalization)} $c\left(B^{2n}(1), \omega_0\right)=\pi= c\left( Z, \omega_0\right)$.
\end{enumerate}
 Here, $\left(B^{2n}(1),\omega_0 \right)$ denotes a ball of radius $1$ in $\R^{2n}$, and $\left( Z= B^2\times \R^{2n-2},\omega_0\right) $ denotes a
 the cylinder of radius $1$ in $\R^{2n}$, where each is equipped with the standard symplectic structure $\omega_0=\sum_{j=1}^n \dd x_i\wedge \dd y_j$ . Before moving on we would like to mention that normalized symplectic capacities are intrinsically difficult to compute since the existence of a normalized symplectic capacity automatically leads to a proof of Gromov's non squeezing theorem.  Nowadays there are several normalized symplectic capacities, c.f.\ \cite{CHLS}, so it is quite natural to ask on which class of domains which symplectic capacities agree?   
This question is closely related to two of the biggest unsolved problems of symplectic geometry:
the \emph{Viterbo conjecture} and the \emph{strong Viterbo conjecture}. 
\begin{Conjecture}[Strong Viterbo conjecture]\label{Strong Viterbor conj}
    Let $X\subseteq\R^{2n}$ be a convex domain. Then all normalized symplectic capacities agree on $X$.
\end{Conjecture}
\noindent
Recent work of Abbondandolo--Benedetti--Edtmair \cite{ABE24} proves that the strong Viterbo conjecture holds true in a $C^2$-neighborhood of the ball in $\R^{2n}$. This extends earlier work of Edtmair \cite{E23di} about the strong Viterbo conjecture in a $C^3$-neighborhood of the ball in $\R^{4}$. Let us also mention the recent work of Cristofaro--Gardiner--Hind \cite{CGH23} that provide a positve answer to the strong Viterbo conjeture for all monotone
toric domains of any dimension. For a general overview we refer to Gutt--Hutchings--Ramos \cite{GHR23}.  Furthermore note that the strong Viterbo conjecture implies: 
\begin{Conjecture}[Viterbo conjecture]\label{Viterbo conj}
  Let $X\subseteq \R^{2n}$ be a convex domain. Then any normalized symplectic capacity $c$ satisfies the inequality \[
  c(X)\leq \left(n! \ \mathrm{vol}(X) \right)^{\frac{1}{n}}.
  \]  
\end{Conjecture}
\noindent
Note that by the seminal work of Artstein--Avidan--Karasev--Ostrover \cite{AKO14} the \emph{Viterbo conjecture} implies the famous \emph{Mahler conjecture}, an old conjecture in convex geometry.  In addition, we note at this point that the Gromov width $c_G$ fulfills the Viterbo conjecture, as it is volume-sensitive. From a topological point the next natural class of symplectic manifolds to investigate the behavior of normalized symplectic capacities are unit disk bundles $D^{*}M$. \\
In this paper, we study two normalized symplectic capacities the Gromov width $c_G$ and the Hofer-Zehnder capacity $c_{\hz}$ on unit disk bundles over lens spaces $L(p;1)$. The Gromov width \( c_G \) of a symplectic manifold \( (M,\omega) \) measures the size of the largest symplectically embedded ball in \( (M,\omega) \), i.e.,  
\[
c_G(M,\omega):= \sup \left\{\pi r^2 \mid \text{ there exists a symplectic embedding of } B^{2n}(r) \text{ into } (M,\omega) \right\},
\]
where \( B^{2n}(r) \) denotes the ball of radius \( r \) in $\R^{2n}$. \\
In contrast, the Hofer--Zehnder capacity \cite[Ch.\ 3]{HZ94} measures the size of a symplectic manifold in terms of the Hamiltonian dynamics on it. It is defined as follows:
$$
c_{\hz}(M,\omega):=\sup\left\lbrace \mathrm{max}(H)\ \vert\ H: M\to \R\ \  \text{smooth, admissible}\right\rbrace,
$$
where admissible means:
\begin{itemize}
    \item[(1)] $0\leq H$ and there exists an open set $U\subset M\setminus\partial M$ such that $H\vert_U\equiv 0$,
    \item[(2)] there exists a compact set $K\subset M\setminus\partial M$ such that $H\vert_{M\setminus K}\equiv\max(H)$,
    \item[(3)] all non-constant periodic solutions $\gamma:\R \to M$ of $\dot \gamma=X_H$ have period $T>1$.
\end{itemize}
We are interested in the following set up. Let $(N,g)$ be a Riemannian manifold and denote
$$
D_1N:=\lbrace (x,v)\in TN\ \vert \ g_x(v,v)<1\rbrace
$$
the unit disk bundle over $(N,g)$. As symplectic structure we consider the metric pullback of the canonical symplectic structure on $T^*N$ and denote it by $\dd\lambda$, i.e.\ $\lambda_{(x,v)}(\xi)=g_x(v,\dd\pi (\xi))$ for all $\xi\in T_{(x,v)}TN$. Although (co-)tangent bundles are a very classical class of symplectic manifolds, the Hofer--Zehnder capacity is only known in a few examples, namely real and complex projective space \cite{B23}, flat tori \cite{BBZ23}, and some convex subsets of \( \mathbb{R}^n \) \cite{AKO14}. In this paper, we add (some) lens spaces to this list. \\
Before stating the main theorem of this paper, we recall for the convenience of the reader the definition of the lens space \( L(p,1) \). Fix an odd number \( p \in \mathbb{N} \) and denote by  
\begin{equation}\label{e: defi lens space}
    L(p,1) = S^3 / \mathbb{Z}_p
\end{equation}
the lens space obtained from the 3-sphere by quotienting out the following \( \mathbb{Z}_p \)-action:  
\[
x \mapsto e^{\frac{2\pi}{p} j} x,
\]
where  $j = i \ominus i = \begin{pmatrix} i & 0 \\ 0 & -i \end{pmatrix}$. Here, we identify \( \mathbb{R}^4 \cong \mathbb{C} \times \mathbb{C} \).

\begin{theorem}\label{thm}
  For any odd $p\in\N$ look at $L(p;1)$ equipped with the metric induced from the round metric on $S^3$, denote by $l$ the length of the shortest closed geodesics. Then the Hofer-Zehnder capacity of $D_1L(p;1)$ equipped with the canonical symplectic form is given by
  $$
    c_{\hz}(D_1L(p;1),\dd\lambda)=pl= 2\pi.
  $$
\end{theorem}\noindent
 So far, in all examples where both Gromov width and Hofer--Zehnder capacity of disk tangent bundles are known (e.g.\ $S^2, \rp^2$ and flat tori, see \cite{BBZ23, B23, FR22}), they agreed. This theorem however shows that for $p$ large enough the Hofer--Zehnder capacity cannot agree with the Gromov width. This is easy to see as the Gromov width is volume sensitive i.e.\ must satisfy
   $$
    c_G(D_1L(p;1))\leq \left( 3! \mathrm{vol}(D_1L(p;1))\right)^{\frac{1}{3}}=\left( 3! \mathrm{vol}(D_1S^3)/p)\right)^{\frac{1}{3}}
    $$
    while the Hofer--Zehnder capacity of $D_1L(p;1)$ is by \Cref{thm} constant in $p$.\footnote{At first it might not look like it, but the length of the shortest geodesics on $L(p;1)$ is the length of prime geodesics on $S^3$ divided by $p$.} So, in particular, as a corollary of \Cref{thm}, we proved that neither the analog of the Viterbo conjecture nor the analog of the strong Viterbo conjecture can hold for disk bundles: \begin{Corollary}
        Let $(D_1 N,\dd\lambda)$ be the unit disk bundle over the n-dimensional Riemannian manifold $(N,g)$ equipped with the canonical symplectic structure. Then  \begin{enumerate}
            \item not all normalized symplectic capacities coincide on $(D_1 N,\dd\lambda)$,
            \item not all normalized symplectic capacities $c$ satisfy the inequality \[
            c\left(D_1 N,\dd\lambda \right)\leq \left(n!\ \mathrm{vol}(D_1 N,\dd\lambda) \right)^{\frac{1}{n}}.
            \]
        \end{enumerate}
    \end{Corollary}
\begin{Remark}
In particular the Hofer--Zehnder capacity of $D_1L(p;1)$ does not obey any volume constraint, which is not too surprising as Usher showed that many closed symplectic manifolds have infinite Hofer--Zehnder capacity \cite{U12}. 
\end{Remark}
\noindent
\Cref{thm} is also interesting, when compared to a relative version of the Hofer--Zehnder capacity. For a closed submanifold $\Sigma\subset M$ not intersecting the boundary, i.e.\ $\Sigma\subset M\setminus \partial M$, the Hofer--Zehnder capacity relative to $\Sigma$ is defined as
$$
c_{\hz}(M,\Sigma,\omega):=\sup\left\lbrace \mathrm{max}(H)\ \vert\ H: M\to \R\ \  \text{smooth, admissible and vanishing on }\Sigma\right\rbrace.
$$
Then for any homotopy class $\nu\in \pi_1(M)$ we define $c_{\hz}^\nu(M,\Sigma,\omega)$ by replacing condition (3) with condition
\begin{itemize}
    \item[(3)'] all non-constant periodic solutions $\gamma:\R \to M$ with $[\gamma]=\nu$ of $\dot \gamma=X_H$ have period $T>1$.
\end{itemize}
\noindent
An important result by Weber \cite{Wbr06} shows that for closed $N$ and a non-zero class $\nu\in\pi_1(N)$, the Hofer-Zehnder capacity relative to the zero section is given by the length $l_\nu$ of the shortest closed geodesic in the class $\nu\in\pi_1(M)$, i.e.
\begin{equation*}
    c_{\hz}^\nu(D_1N,N,\dd\lambda)=l_\nu.
\end{equation*}
For non-aspherical, homogeneous spaces and positive curvature metrics on the 2-sphere the result was extended to the class of contractible loops ($\nu=0$) by Benedetti and Kang \cite[Cor.\ 2.8]{BK22}. In particular if either the shortest geodesic is non-contractible or $N$ is non-aspherical, homogeneous spaces or a positive curvature 2-sphere we can deduce
\begin{equation}\label{relative Hofer-Zehnder capacity}
c_{\hz}(D_1N,N,\dd\lambda)=l,
\end{equation}
where $l$ denotes the length of the shortest closed geodesic. This is because $c_{\hz}^\nu(D_1N,N,\dd\lambda)$ yields an upper bound for any homotopy class $\nu$ and the Hamiltonian $H(x,v):=l\vert v\vert$ can easily be modified, without changing the oscillation much, so that it becomes admissible.\\
\noindent
In \cite{B23} it was shown that for $N=\cp^n$ the Hofer--Zehnder capacity coincides with relative Hofer--Zehnder capacity, while for $N=\rp^n$ there is a factor of 2. Using \Cref{thm} and \eqref{relative Hofer-Zehnder capacity} we can extend the result for $N=\rp^3$ to the following class of lens spaces $L(p;1)$:
\begin{Corollary}\label{relative Hofer Zehnder Disk bundle lens space} For odd $p\in\N$, the Hofer-Zehnder capacity and the relative Hofer-Zehnder capacity of $D_1L(p;1)$ are in the following ratio to each other: 
    \begin{equation*}
        c_{\hz}(D_1L(p;1),\dd\lambda)=p\cdot c_{\hz}(D_1L(p;1),L(p;1),\dd\lambda).
    \end{equation*}
\end{Corollary}
\ \\
\noindent
\textbf{Outline:}
The proof of Theorem \ref{thm} has two steps. The first is finding a lower bound. For this we construct a nice Hamiltonian and determine its flow and in particular its periodic orbits explicitly. The dynamics that will give the correct lower bound lift to magnetic billiard dynamics on the 3-sphere with the Hopf-link removed. The construction of this flow is carried out in section 2. The second step is finding an upper bound that coincides with the lower bound. For the upper bound we use that
$$
c_{\hz}(DL(p;1),\dd\lambda)\leq c_{\hz}(DS^3,\dd\lambda),
$$
as admissible Hamiltonians on $DL(p;1)$ lift to admissible Hamiltonians on $DS^3$. An upper bound for the capacity of $DS^3$ is then obtained from results by Hofer--Viterbo \cite{HV92} and Lu \cite{Lu06}, which heavily relies on pseudoholomorphic curve techniques. This works because the Lerman cut \cite{ler} with respect to the geodesic flow $\overline{(DS^3,\dd\lambda)}$ is symplectomorphic to the complex quadric $Q^3:=\lbrace [z_0:\ldots:z_4]\ \vert\ z_0^2+\ldots + z_4^2=0\rbrace\subset \cp^4$ with Fubini-Study form. Hence Gromov-Witten invariants are explicitly computable.\\
\ \\
\noindent
\textbf{Acknowledgments.}\\
The authors are grateful to G. Benedetti, D. Cristofaro-Gardiner, R. Hind and S. Tabachnikov for helpful comments on a earlier draft of this paper. The authors are also grateful to their old friend G. Mo\u gol, who helped them out with his Python skills. We are indebted to the referee for their comments, which helped improve the paper. The authors further acknowledges funding by the Deutsche Forschungsgemeinschaft (DFG, German Research Foundation) – 281869850 (RTG 2229), 390900948 (EXC-2181/1) and 281071066 (TRR 191). L.M. gratefully acknowledges support from the Simons Center for Geometry and Physics, Stony Brook
University at which some of the research for this paper was performed during the program
\textit{Mathematical Billiards: at the Crossroads of Dynamics, Geometry, Analysis, and Mathematical Physics.} L.M. thanks also V. Assenza and V. Ramos for their warm hospitality at IMPA in Rio de Janeiro in February 2024 where some of the research for this paper was performed.  L.M. thanks the participants of the conferences \textit{Symplectic Dynamics in Montevideo} in Montevideo in February 2024 for helpful feedback on preliminary versions of this work.\\

\section{Lower bound: magnetic billiards on lens spaces}
Consider $S^3:=\lbrace \vert x \vert^2=1\rbrace \subset \R^4$, where $\vert \cdot\vert$ denotes the standard norm on $\R^4$. 
In order to obtain the correct lower bound for the Hofer--Zehnder capacity of \( D_1 S^3 \), one can use the Hamiltonian \( H(x,v) := 2\pi |v| \). The Hamiltonian flow on \( D_1 S^3 \) generated by \( H \) is a reparametrization of the geodesic flow on \( S^3 \) equipped with the round metric, in such a way that all orbits are closed and have period one. Now modifying $H$ a little near the zero-section and the boundary one obtains an admissible Hamiltonian that has ossicillation arbitrarily close to $1\cdot 2\pi$, hence gives the correct lower bound to prove Theorem \ref{thm} in the case of $p=1$. For lens spaces the situation will be more complicated as prime geodesics can have different lengths. Let us examine the setup more closely. Let \( L(p,1) \) be as in \eqref{e: defi lens space}, and denote by \( g \) the metric induced from the round metric on \( S^3 \).  
\begin{Lemma}\label{l1}
    All prime geodesics on \( (L(p,1), g) \), except for those in the direction of \( jx \), have length \( 2\pi \). The prime geodesics in the direction of \( jx \) have length \( 2\pi / p \).
\end{Lemma}

\begin{proof}
Take a point $x\in S^3$, now all points $e^{2\pi j q/p}x$ identified with $x$ under the $\Z_p$ action lie on the geodesic in direction $jx$. As $p$ is odd for all $q$ the geodesic in direction $jx$ is the unique geodesic through $x$ and $e^{2\pi jq/p}x$. 
\end{proof}
\noindent
Observe that this means simply reparametrizing the geodesic flow, as in the case of $S^3$, is not enough to obtain the correct lower bound. Instead, we want to introduce a small magnetic perturbation to destroy almost all of the short geodesics.
\noindent
By abuse of notation, we write $i=i\oplus i$ for the standard complex structure on $\C^2$. This complex structure induces the standard contact form on $S^3$, namely
$$
\alpha_x: T_xS^3\to \R;\ v\mapsto \frac{1}{2}\mathrm{Re}\langle ix,v \rangle.
$$
Here $\langle \cdot, \cdot\rangle$ denotes the standard Hermitian product on $\C^2$ with respect to $i$. To find a lower bound of the Hofer--Zehnder capacity we need to explicitly construct an admissible Hamiltonian. The strategy for constructing such a Hamiltonian is as follows: \begin{enumerate}
    \item We will use the magnetic geodesic flow on $\left(S^3, g, \varepsilon\cdot\dd\alpha \right)$ to destroy almost all short geodesics except two, $\gc_{\pm}$ parametrizing the Hopf link (see Lemma \ref{lemma lense space}).
    \item After removing the Hopf link $L=\gc_{+}\cup \gc_{-}$ in $S^3$ and adding a $\Z_p$-invariant potential, which tends to infinity near $\gc_{\pm}$, to the kinetic Hamiltonian the induced Hamiltonian flow on $L(p;1)$ has no fast periodic orbits (see Corollary \ref{periodlens}).
    \item The Hamiltonian can be modified to be admissible, and its oscillation is given by $pl$, which yields a lower bound on the Hofer-Zehnder capacity (see Corollary \ref{lower bound Hofer Zehnder capacity}).
\end{enumerate}
\noindent
\textbf{Intro: Magnetic systems}\\
Before we start with the construction of the lower bound we recall basic facts about magnetic systems, for a more detailed introduction look at \cite{Bd14} and the references there in. The triple $(M,g,\sigma)$ is called a \emph{magnetic system}, where $M$ is a closed manifold, $g$ a Riemannian metric on $M$ and $\sigma\in\Omega^2(M)$ a closed  $2$-form on $M$. This data gives rise to a twisted symplectic form $\omega_{\sigma}=\dd\lambda-\pi^*\sigma$ on the tangent bundle $TM$, where $\pi$  denotes the canonical projection of $TM$ onto $M$ and $\lambda$ the pullback of the canonical symplectic form on $T^*M$ to $TM$ via the duality isomorphism given by $g$. The symplectic form $\omega_{\sigma}$ is nondegenerate so there exist a unique vector field $X_E^{\sigma}$ on $TM$, called the Hamiltonian vector field of $(E, \omega_{\sigma})$, such that $$\omega_{\sigma}\left( X_{E}^{\sigma}, \cdot\right)=\dd E(\cdot),$$ where $E(x,v)=\frac{1}{2}g_x(v,v)$ denotes the kinetic Hamiltonian on $TM$. The flow associated to $X_E^{\sigma}$ is called the Hamiltonian flow of $E$ and is given by \[
\varPhi_{E, \sigma}^t:TM \longrightarrow TM:(x,v)\mapsto \Gamma_{(x,v)}(t)
\]
where $\Gamma_{(x,v)}(t)$ is the unique solution of $\dot{\Gamma}=X_{E}^{\sigma}$ with $\Gamma(0)=(x,v)$. Since the derivative of $E$ vanishes on the energy surface \[\Sigma_k=\{(x,v)\in TM: E(x,v)=k\}=E^{-1}(k)\subset TM,\] with $k\in[0,\infty)$, these submanifolds are invariant by the Hamiltonian flow $\varPhi_{E, \sigma}$. By \cite{Gin} flow lines are of the form $\Gamma=(\gamma,\dot\gamma)$, where $\gamma:\R\to M$ is a magnetic geodesic on $(M, g, \sigma)$ i.e.\ a solution to \begin{equation}\label{general magnetic geodesic equation}
	\nabla_{\dgc}\dgc= Y_{\gc}\dgc,
\end{equation}
where $Y:TM\longrightarrow TM$, called the \emph{Lorenz force}, is the unique bundle endomorphism which satisfies
\[
g_p\left(Y_p(v), (w)\right)=\sigma_p(v, w) \quad \forall (p,v)\in TM. 
\] 
Using this duality between magnetic geodesic on $(M,g,\sigma)$ and flow lines of $X_{E}^{\sigma}$ and the fact that the energy surfaces $\Sigma_k \subset TM$ are invariant by Hamiltonian flow $\varPhi_{E, \sigma}$ we see that if $\gc$ is a magnetic geodesic on $(M,g,\sigma)$, then its kinetic energy is a conserved quantity i.e. 
\begin{equation}\label{kinetic energy is integral of motion}
	\frac{1}{2}g_{\gc(t)}(\dgc(t), \dgc(t))=E(\gc(t),\dgc(t))= E(\gc(0), \dgc(0))=\frac{1}{2}g_{\gc(0)}(\dgc(0), \dgc(0))\quad \forall t\in \R.
\end{equation}
In addition, it is quite natural to ask how the dynamics of the Hamiltonian flow $\varPhi_{E, \sigma}$ changes if we deform the twisted symplectic form as follows \begin{equation*}
    \left(\omega_{\sigma}\right)_{\varepsilon}=\dd \lambda-\varepsilon\pi^*\sigma \textit{ for }\varepsilon\in\R.
\end{equation*}
Thus a deformation of $\varepsilon$ induce a deformation of the Hamiltonian vector field $\left(X_{E,\sigma}\right)_{\varepsilon}$, where $\left(X_{E,\sigma}\right)_{\varepsilon}$ is defined through\begin{equation*}
    \left(\omega_{\sigma}\right)_{\varepsilon}\left( \left(X_{E,\sigma}\right)_{\varepsilon}, (\cdot)\right)=\dd E(\cdot).
\end{equation*} Note also that for $\varepsilon=0$, the vector field $\left(X_{E,\sigma}\right)_0$ is exactly the geodesic vector field, thus its associated Hamiltonian flow is the geodesic flow of $(M,g)$. By the duality mentioned before \eqref{general magnetic geodesic equation}, if the curve $\Gamma_{\varepsilon}$ is a flow line of the Hamiltonian vector field $\left(X_{E,\sigma}\right)_{\varepsilon}$, then its canonical projection $\pi\left(\Gamma_{\varepsilon}\right)$ is a magnetic geodesic on $(M,g,\sigma)$ of strength $\varepsilon$, i.e., it is a solution of 
 \begin{equation}\label{general equation for magnetic geodesics of strength s}
	\nabla_{\dgc}\dgc=\varepsilon Y_{\gc}\dgc.
\end{equation} 
If we make use of \eqref{kinetic energy is integral of motion} it follows directly that if $\gc$ is a magnetic geodesic of strength $\varepsilon$, then its kinetic energy is an integral of motion. Furthermore, if the magnetic field $\sigma\in \Omega^2(M)$ is exact, i.e.\ there exists a $\beta\in \Omega^1(M)$ such that $ \dd\beta=\sigma$, there is another point of view, the \emph{Lagrangian point of view}. To be precise the curve $\gc$ is a magnetic geodesic on $(M, g, \sigma)$ of strength $\varepsilon$, if and only if $\gc$ is a solution of the Euler--Lagrange equations associated to the Lagrangian 
\begin{equation*}
    L:TM\longrightarrow\R:(x,v)\mapsto\frac{1}{2}g_x(v,v)+ \varepsilon \beta_x(v). 
\end{equation*}
\noindent
\textbf{Step 1: Magnetic geodesic flow on $L(p;1)$}\\
Now we will focus on a more specific magnetic system: $\left(S^3, g, \varepsilon\cdot\dd\alpha \right)$, the 3-sphere equipped with the round metric $g$ and the standard contact form $\alpha$ as a magnetic field.  To construct the lower bound of the Hofer-Zehnder capacity we will use this magnetic system. This dynamical system was studied in full detail by Albers-Benedetti and the second author in \cite{ABM}. By following precisely the lines of \cite[Eq. (3.2)]{ABM}  we see that a magnetic geodesic $\gamma$ starting at $\gamma(0)=x\in S^3$ in direction $\dot \gamma (0)=v\in T_xS^3$ is a curve which is a solution of the following differential equation
\begin{equation}\label{e1}
\ddot \gamma -i\varepsilon\dot \gamma +(c^2 -\varepsilon \delta)\gamma =0,
\end{equation}
in the ambient space $\C^2$, where we used that \( c=\vert \dot{\gamma} \vert=\vert v \vert \) and \( \delta=\mathrm{Re} \langle i\gamma, \dot{\gamma} \rangle = \mathrm{Re} \langle ix, v \rangle = \cos(\psi) |v| \) are by \cite{ABM} integrals of motion. Here, \( \psi \) represents the angle between the Reeb vector field \( R_z \) and a tangent vector \( (z,v) \in TS^3 \).
 By the previous discussion we see that $\gc$ is magnetic geodesic on $(S^3,g, \varepsilon\dd\alpha)$ if and only if $(\gc, \dgc)$ is a flow line of the Hamiltonian system $(TS^3, \dd\lambda-\varepsilon\pi^*\dd\alpha, E)$.  
As the magnetic form is exact we can shift the zero-section to symplectically embed the twisted tangent bundle into the standard tangent bundle
\begin{equation}\label{shift of zero section}
\iota: (D_{1-2\varepsilon}S^3, \dd\lambda-\varepsilon\pi^*\dd\alpha)\hookrightarrow (D_1 S^3,\dd\lambda):\ (x,v)\mapsto (x,v+\varepsilon ix).
\end{equation}
To give a lower bound on the Hofer--Zehnder capacity of $(D_1 S^3,\dd\lambda)$ we will first give a lower bound on the periods of orbits of the Hamiltonian system $(D_{1-2\varepsilon}S^3, \dd\lambda-\varepsilon\pi^*\dd\alpha, E)$: 
\begin{Lemma}\label{Lower bound for magnetic geodesics on S3}
    All non-constant periodic orbits of the kinetic Hamiltonian
    $$
    E:D_{1-2\varepsilon}S^3\to\R;\quad (x,v)\mapsto \frac{1}{2}\vert v\vert_x^2
    $$
    have period 
    $$
    T_E\geq \begin{cases} \frac{2\pi}{\vert v\vert} \ \ \text{for} \ \vert v\vert\geq\varepsilon\\
     \frac{2\pi}{\varepsilon} \ \ \text{for} \ \vert v\vert\leq\varepsilon.
    \end{cases}
    $$
\end{Lemma}
\begin{proof}
In \cite{ABM} it is shown that the orbits stay on Clifford tori, spiraling around the Reeb direction. Indeed there is an explicit formula for general magnetic geodesics
    $$
    \gamma(t)=e^{i\theta_+t}p_++e^{i\theta_-t}p_-,
    $$
    where
    $$
    \theta_\pm=\frac{\varepsilon\pm\sqrt{\varepsilon^2+4(c^2-\varepsilon\delta)}}{2}\ \ \& \ \ p_\pm=\mp\frac{\theta_\mp x+iv}{\theta_+ -\theta_-}.
    $$
    Further $c:=\vert\dot \gamma\vert=\vert v\vert$ and $\delta:= \mathrm{Re}\langle i\gamma,\dot\gamma\rangle=\mathrm{Re}\langle ix,v\rangle$ are integrals of motion.
    If $\theta_-\neq 0$ it follows immediately that
    $$
    T_E\geq \frac{2\pi}{\vert \theta_-\vert}\geq \frac{2\pi}{c},
    $$
    where we used that
    $$
    c=\sqrt{\theta_+^2\vert p_+\vert^2+\theta_-^2\vert p_-\vert^2}\geq \sqrt{\theta_-^2\left(\vert p_+\vert^2+\vert p_-\vert^2\right)}=\vert \theta_-\vert.
    $$
    If on the other hand $\theta_-=0$, we find $\gamma(t)=e^{i\varepsilon t}p_++p_-$. It easily follows that
    $$
    T_E=\frac{2\pi}{\varepsilon}.
    $$
    As $\theta_-=0$ implies $c^2=\varepsilon\delta$ and $\delta<c$ we have $c<\varepsilon$.
\end{proof}
\noindent
 The advantage of the magnetic geodesic flow becomes clear when studying the lens space. The idea is that the magnetic term bends the short flow lines (just an $\varepsilon$-bit) away from the geodesic that is contracted by the quotient $S^3\to L(p;1)$. Now lets look at the details. First, we need to make sure that the magnetic term $\alpha$ actually descends to the quotient. Hence, we can think of $\alpha\in \Omega_1(L(p;1))$.
\begin{Lemma}\label{Magnetic flow is invariant under Zp}
    The 1-form $\alpha\in\Omega_1(S^3)$ is $\Z_p$-invariant, i.e.
    $$
    \left(e^{\frac{2\pi}{p}j}\right)^*\alpha_{x}=\alpha_{e^{-\frac{2\pi}{p}j} x}.
    $$
\end{Lemma}
\begin{proof} Observe that $ji=ij$, hence
    $$
    \left(e^{\frac{2\pi}{p}j}\right)^*\alpha_{x}(\cdot)=\frac{1}{2}\langle ix, e^{\frac{2\pi}{p}j}\cdot\rangle=\frac{1}{2}\langle  e^{-\frac{2\pi}{p}j}ix,\cdot\rangle=\frac{1}{2}\langle i e^{-\frac{2\pi}{p}j}x,\cdot\rangle=\alpha_{e^{-\frac{2\pi}{p}j} x}.
    $$
\end{proof}
\noindent
This means that the magnetic twist descends to the lens space, but the magnetic geodesic flow on $L(p;1)$ has (at least) two short periodic orbits $\gamma_+(t)=(e^{2\pi it},0)$ and $\gamma_-(t)=(0,e^{2\pi it})$, as $j\gamma_\pm=\pm i\gamma_\pm$ and magnetic geodesics in Reeb direction $ix$ are ordinary geodesics.\\
\ \\
\textbf{Intermezzo: Alternative description of magnetic geodesics on $S^3$}\\
In order to finish Step 1 we need to prove that these two are the only short periodic orbits of the magnetic flow on $L(p;1)$. The following alternative description of magnetic geodesics will be useful. It is based on the geometric intuition that one can think of a magnetic geodesic as a composition of rotation in Reeb direction and a rotation of some special complex structure, to be precise: 
\begin{Proposition}\label{p1}
    A solution $\gamma$ of \eqref{e1} with initial conditions $ \gamma(0)=x\in S^3$ and $\dot\gamma(0)=v \in T_xS^3$ can be written as
    $$
    \gamma(t)=e^{\frac{i\varepsilon}{2}t}e^{J(x,v,\varepsilon)at}x
    $$
    where $a=\frac{1}{2}\sqrt{\varepsilon^2+4(c^2-\varepsilon\delta)}$ and $J(x,v,\varepsilon)$ is the unique complex structure, preserving the standard orientation and round metric on $S^3$ that satisfies
    $$
    J(x,v,\varepsilon)x=\frac{1}{a}\left(v-i\frac{\varepsilon}{2}x\right).
    $$
\end{Proposition}
\noindent
Before we prove the proposition, observe that $J(x,v,\varepsilon)$ is indeed unique, as orientation determines it on the orthogonal complement of $\mathrm{span}\lbrace x, v-i\frac{\varepsilon}{2}x\rbrace$. Explicitly we set $J(x,v,\varepsilon)e_1=e_2, J(x,v,\varepsilon)e_2=-e_1$ for $(e_1,e_2)$ any ordered (hence oriented) orthonormal bases of the orthogonal complement of $\mathrm{span}\lbrace x, v-i\frac{\varepsilon}{2}x\rbrace$. Further, $J(x,v,\varepsilon)$ belongs to the 2-sphere of complex structures (commuting with $i$) that comes from identifying $\C^2$ as the quaternionic space.
\begin{proof}
    We simply check that $\gamma(t)=e^{\frac{i\varepsilon}{2}t}e^{J(x,v,\varepsilon)at}x$ satisfies \eqref{e1}. We abbreviate $J=J(x,v,\varepsilon)$ and compute
    $$
    \dot\gamma=\left(i\frac{\varepsilon}{2}+Ja\right)\gamma
    $$
    and
    $$
    \ddot \gamma=\left(i\frac{\varepsilon}{2}+Ja\right)^2\gamma=\left( -\frac{\varepsilon^2}{4}+iJ\varepsilon a-a^2\right)\gamma.
    $$
    This indeed implies 
    \begin{align*}
        \ddot \gamma -i\varepsilon\dot \gamma +(c^2 -\varepsilon \delta)\gamma&=\left( -\frac{\varepsilon^2}{4}+iJ\varepsilon a-a^2-i\varepsilon \left(i\frac{\varepsilon}{2}+Ja\right)+c^2 -\varepsilon \delta\right)\gamma\\
        &=\left ( \frac{\varepsilon^2}{4}-a^2+c^2-\varepsilon\delta\right)\gamma=0.
    \end{align*}
\end{proof}
\noindent
This description of the magnetic geodesics yields a nice visualization (Figure \ref{fig1}) in terms of the Hopf-fibration $\pi: S^3\to S^2$ (for $i$). The factor $e^{\frac{i\varepsilon}{2}t}$ is a rotation in the fibers, while the factor $e^{Jat}$ projects down to a rotation of $S^2$ with rotation axis determined by $J$. Explicitly, the two fixed antipodal points $\pi(x_\pm)$ of the rotation are the projections of solutions to the equation $J x_\pm =\pm i x_\pm$.
\begin{figure}[]
	\centering
 \includegraphics[width=0.3\textwidth]{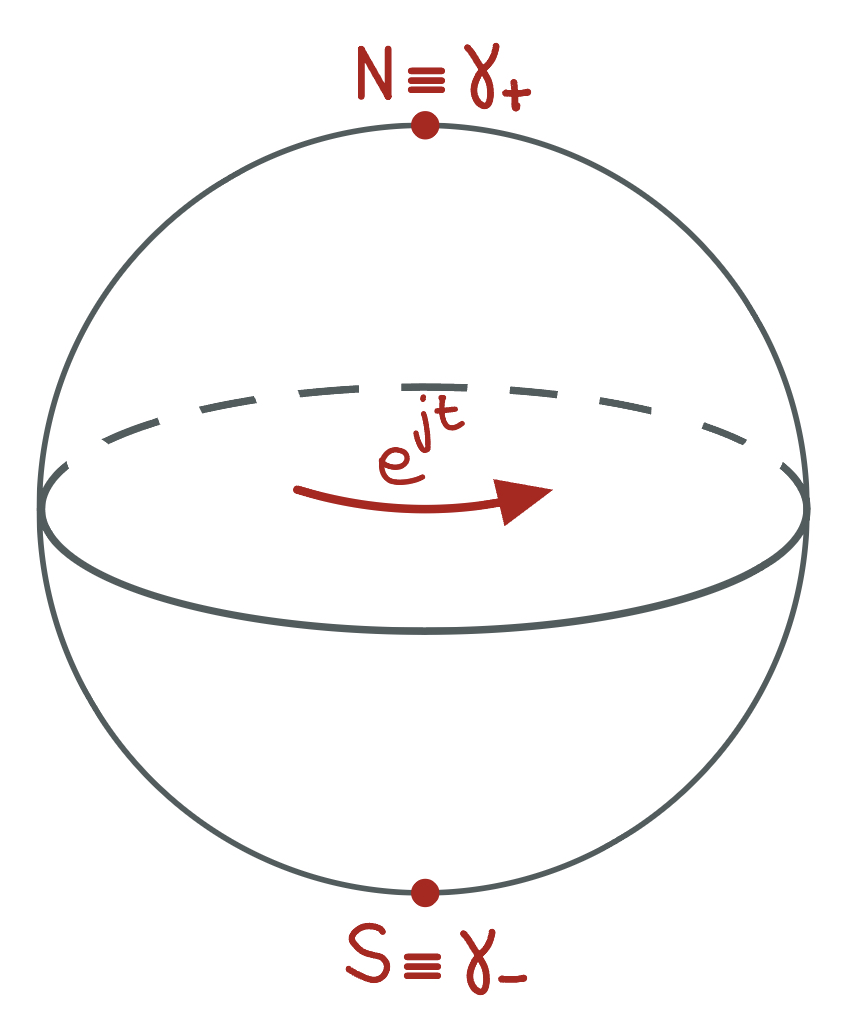}
	\caption{\textit{The figure shows the image of the Hopf projection $\pi: S^3\to S^2$ with respect to the complex structure $i \equiv i\oplus i$ on $\R^4\cong \C^2$. The orbits $\gamma_\pm$ that satisfy $j\gamma_\pm=\pm i\gamma_\pm$ are projected to antipodal points we call $N$ and $S$. The rotation $e^{jt}$ is projected to the rotation that fixes $N$ and $S$. Similarly any rotation of the form $e^{Jt}$ (for a complex structure $J$ commuting with $i$) projects to a rotation fixing the points with fibers determined by $J\gamma_{\pm}=\pm i\gamma_\pm$.}}
    \label{fig1}
\end{figure}
\noindent
We can now prove that the only fast periodic magnetic geodesics on $L(p;1)$ are the quotients of the Reeb orbits over the antipodal points $N, S$.
\begin{Lemma}\label{lemma lense space}
    All non-constant periodic magnetic geodesics on $L(p;1)$ of period $T$ except for the Reeb orbits $\gamma_\pm$ satisfy 
    $$
    T_E\geq \begin{cases} \frac{2\pi}{\vert v\vert} \ \ \text{for} \ \vert v\vert\geq\varepsilon\\
     \frac{2\pi}{\varepsilon} \ \ \text{for} \ \vert v\vert\leq\varepsilon.
    \end{cases}
    $$
\end{Lemma}
\begin{proof}
    All magnetic geodesics on $L(p;1)$ that lift to closed magnetic geodesic on $S^3$ satisfy the bound by Lemma \ref{Lower bound for magnetic geodesics on S3}. Therefore let $\gamma$ be a magnetic geodesic of period $T$ on $S^3$ that satisfies 
    \begin{equation}\label{e2}
        e^{j\frac{2\pi}{p}}\gamma(t)=\gamma(t+T_E/p),
    \end{equation}
    which is equivalent to projecting to a short magnetic geodesic on $L(p;1)$. As $[j,i]=0$ this implies that $e^{j\frac{2\pi}{p}}$ projects under the Hopf map to a rotation of $S^2$ that preserves the geodesic circle $\pi(\gamma)$, which implies $J(x,v,\varepsilon)=j$. Evaluating equation \eqref{e2} at $t=0$ we find
    $$
    (e^{i\frac{2\pi}{p}}z_1,e^{-i\frac{2\pi}{p}}z_2)=e^{i\frac{\varepsilon}{2}\frac{T_E}{p}}(e^{\frac{iaT_E}{p}}z_1,e^{\frac{-iaT_E}{p}}z_2),
    $$
    where $\gamma(0)=:(z_1,z_2)$. Multiplying both entries yields $e^{i\varepsilon \frac{T_E}{p}}z_1z_2=z_1z_2$ which implies
    $$
    \frac{T_E}{p}=\frac{2\pi k}{\varepsilon}\geq\frac{2\pi}{\varepsilon}
    $$
    unless $z_1=0$ or $z_2=0$ which are precisely the orbits $\gamma_\pm$.
\end{proof}

\noindent
\textbf{Step 2: Adding a potential -- magnetic billiards}\\
In order to construct an admissible Hamiltonian on (an open set of) $D_1L(p;1)$ we need to modify $E$, so that all non-constant periodic orbits have period $T>1$. Recall that  $\gamma_\pm$ are the only periodic orbits that will cause problems, so we will try to cut them out using a potential. Again we first construct the potential on $S^3$ and later argue that it descends to $L(p;1)$. If 
$$
\pi: S^3(1)\to S^2\left (\frac{1}{2}\right)
$$
denotes the Hopf-map with respect to $i$, these two orbits $\gc_{\pm}$ project to antipodal points. We call them $N$ and $S$. Now the rotation $e^{jt}$ projects under $\pi$ to the rotation of $S^2(\frac{1}{2})$ fixing $N$ and $S$ (see Figure \ref{fig1}). We want to modify $E$ by adding a potential $ V_\varepsilon:L(p;1)\setminus \gamma_\pm\to \R$ that only depends on the distance $d(x,\gamma_\pm)$, goes to infinity near $\gamma_\pm$ and is zero outside the $\varepsilon$ neighborhood of $\gamma_\pm$. 

\begin{Proposition} \label{prop potential}
    All periodic orbits of $E+V_\varepsilon$ on $TS^3$ have period
    $$
    T_{E+V_\varepsilon}\geq \begin{cases} \frac{2\pi}{\max\vert v\vert} \ \ &\text{for} \ \max\vert v\vert\geq\sqrt \varepsilon\\
     \frac{2\pi}{\varepsilon} \ \ &\text{for} \ \max\vert v\vert\leq\sqrt\varepsilon
    \end{cases}-\mathcal{O}(\varepsilon).
    $$
    Here $\max\vert v\vert:=\max_t\vert \dot\gamma(t)\vert$ refers to the maximal velocity along the periodic orbit. Observe that $E(\gamma,\dot\gamma)+V_\varepsilon(\gamma)=\frac{1}{2}\max\vert v\vert^2$ as the potential vanishes where the velocity is maximal.
\end{Proposition}
\begin{Notation}
    From now on we will use the term $\mathcal{O}(\varepsilon)$ by abuse of notation as a term that tends to zero as $\varepsilon\to 0$, but is not necessarily of order $\varepsilon$.
\end{Notation}
\noindent
The idea of the proof is very geometric: If $\gamma_\varepsilon$ is a sequence of periodic orbits of $E+V_\varepsilon$, then the $H^1$-limit of $\gc_{\varepsilon}$ as $\varepsilon\to 0$ is either a magnetic bounce orbit on the billiard table $S^3\setminus\gamma_\pm$ or an orbit trapped near the boundary of the billiard table $\gamma_\pm$. (For a precise definition we refer to \Cref{B}.) In particular, when the periods of $\gamma_\varepsilon$ are bounded uniformly in $\varepsilon$, their periods converge to the periods of the bounce orbits.

\begin{figure}[h]
	\centering
 \includegraphics[width=0.8\textwidth]{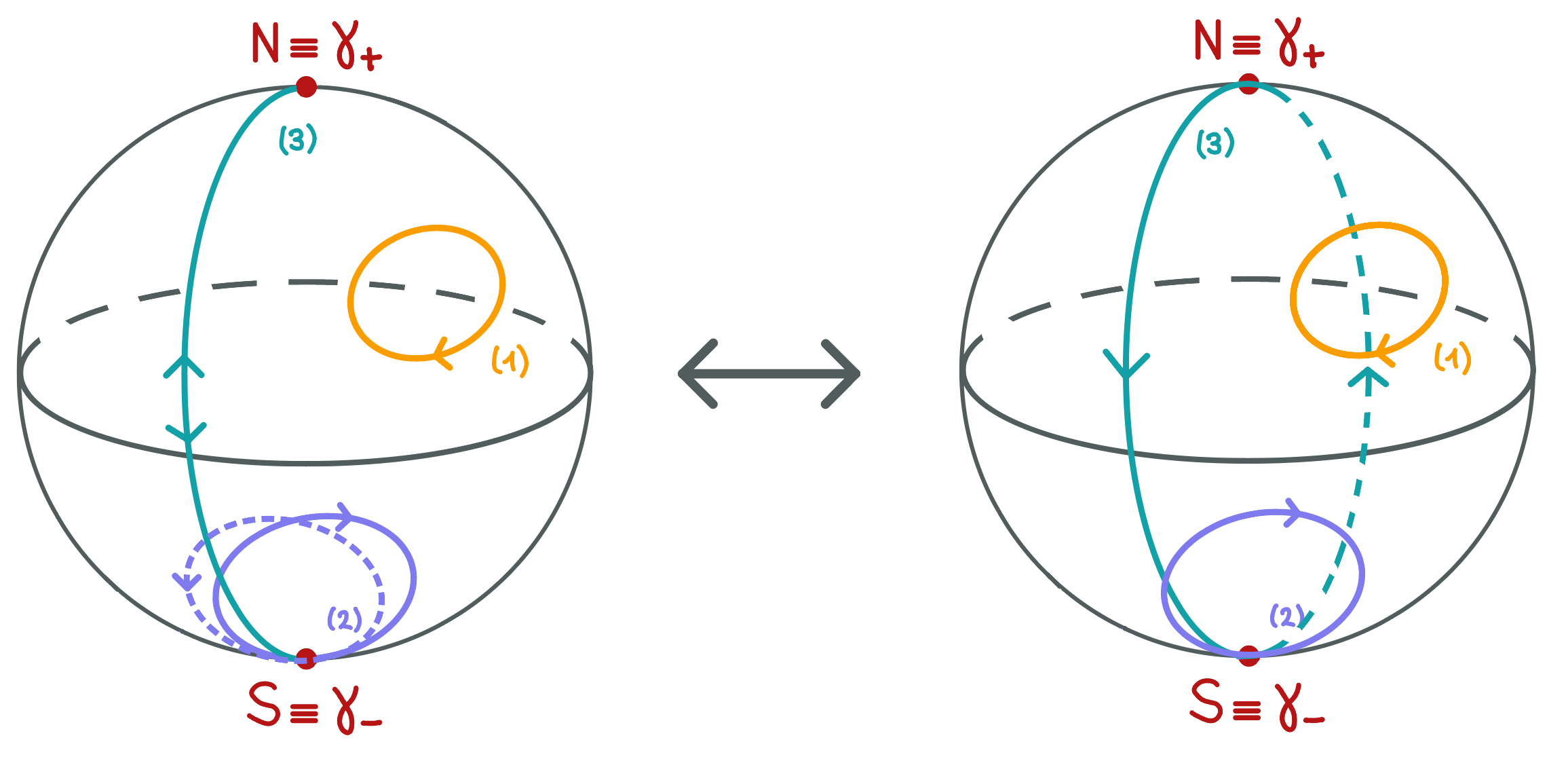}
	\caption{\textit{There are three types of bounce orbits: (1) no bounce, (2) bounce at exactly one pole, (3) bounce at both poles. At every bounce the velocity in contact direction is reversed, while $\delta$, the velocity in Reeb direction, is preserved. In particular as depicted on the right, every periodic bounce orbit corresponds by symmetry to a periodic magnetic geodesic (without bounces). Here symmetry means a reflection at a suitable plane.}}
    \label{types}
\end{figure}

\begin{Lemma}\label{lemma bounce orbits}
    All periodic magnetic bounce orbits $\gamma_b$ on $S^3$ have the same period as a corresponding smooth orbit $\gamma$, i.e. $T_{\text{bounce}}(\gamma_b)=T_E(\gamma)$.
\end{Lemma}
\begin{proof}
    As shown in Figure \ref{types} there are three types of bounce orbits. Type (1) is a smooth orbit and thus satisfies the period bound by Lemma \ref{Lower bound for magnetic geodesics on S3}. Type (2) bounces exactly at one of the poles. If it is periodic, its period is by symmetry the same as the orbit of the corresponding smooth orbit. Type (3) bounces at both poles and again using symmetries one can sees that its period is the same as the period of the corresponding smooth orbit.
\end{proof}
\noindent
The detailed proof of Proposition \ref{prop potential} is written in Appendix \ref{B}. It follows the strategy above arguing that for small $\varepsilon$ any periodic orbit must either be close to a periodic bounce orbit or trapped near the poles. In the first case one obtains the bound from Lemma \ref{lemma bounce orbits}, in the second case from using some integrals of motion that are derived in Appendix \ref{A}.\begin{Remark}
    At this point we also remark that in principle it would be also possible to construct the magnetic bounce orbits with the approximating scheme of magnetic bounce orbits developed by Vocke in her thesis \cite{V21}, which is a generalization of earlier work of Albers-Mazzucchelli \cite{AM11}. 
\end{Remark}
\noindent
The potential $V_\varepsilon$ is $\Z_p$-invariant, therefore the Hamiltonian $E+V_\varepsilon$ descends to the lens space.

\begin{Corollary}\label{periodlens}
    All periodic orbits of $E+V_\varepsilon$ on $TL(p;1)$ have period
    $$
    T_{E+V_\varepsilon}\geq \begin{cases} \frac{2\pi}{\max\vert v\vert} \ \ &\text{for} \ \max\vert v\vert\geq\sqrt\varepsilon\\
     \frac{2\pi}{\sqrt\varepsilon} \ \ &\text{for} \ \max\vert v\vert\leq\sqrt\varepsilon
    \end{cases}-\mathcal{O}(\varepsilon).
    $$
\end{Corollary}
\begin{proof}
We need to check that $\Z_p$-invariant orbits on $S^3$ have period at least $p$ times the above bound. A $\Z_p$-invariant orbit is either a magnetic geodesic as in Lemma \ref{lemma lense space} and therefore satisfy the period bound. Note that for $\varepsilon \leq \max\vert v\vert \leq \sqrt{\varepsilon}$ we have $\frac{2\pi}{\max\vert v\vert}\geq \frac{2\pi}{\sqrt{\varepsilon}}$ and for $\max\vert v\vert \leq\varepsilon$ we have $\frac{2\pi}{\varepsilon}\geq\frac{2\pi}{\sqrt{\varepsilon}}$. So the change of the cutoff from $\varepsilon$ to $\sqrt{\varepsilon}$ does not course any problems. If the $\Z_p$-invariant orbit interferes with the potential (i.e. not purely magnetic geodesic) it is a $\Z_p$-symmetric configuration of magnetic geodesic arcs (see Figure \ref{symmetric}). If it is not trapped near the caps this means it must have length tending to infinity as the bounce orbits are not $\Z_p$-symmetric. If the orbit is trapped near a cap, it also must consists of at least $p$ magnetic geodesic arcs. This means it has winding number at least $p$ in Reeb direction and therefore $T_{E+V_\varepsilon}\geq p\cdot \frac{2\pi}{\max \vert v\vert}$ on $S^3$.
\end{proof}
\begin{figure}[h]
	\centering
 \includegraphics[width=0.35\textwidth]{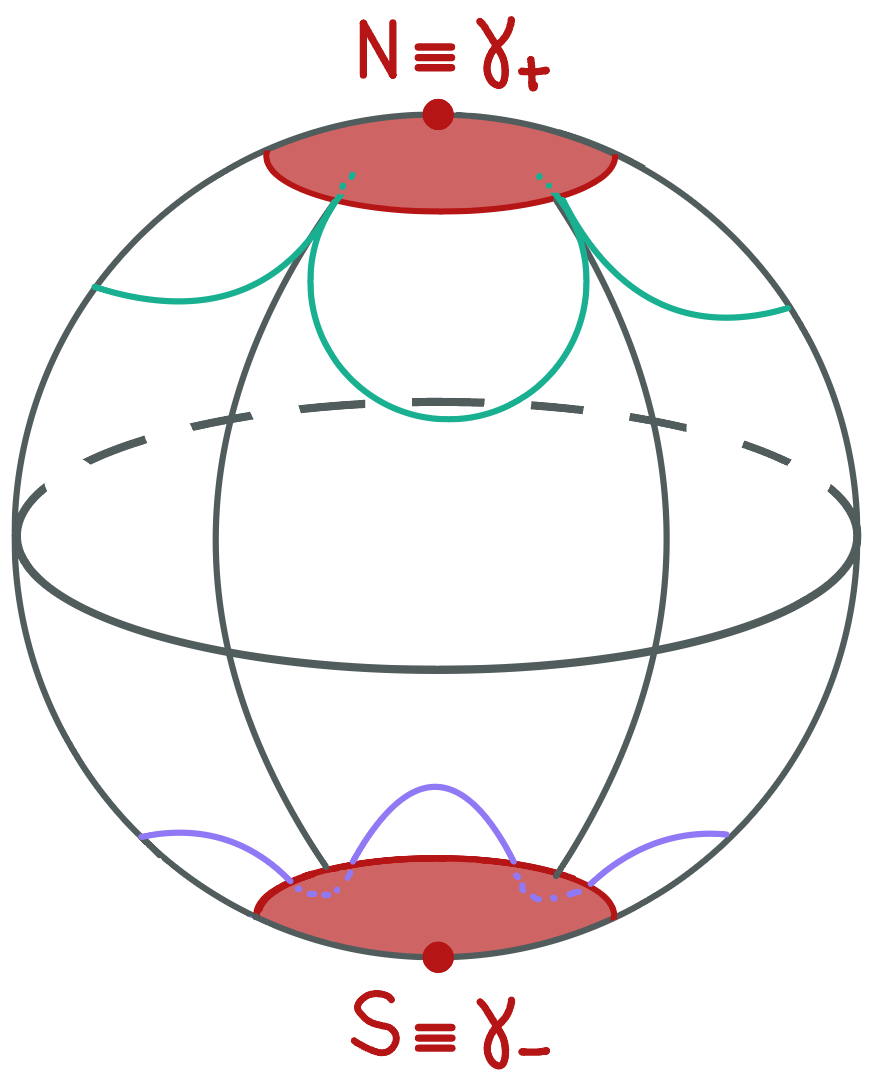}
	\caption{\textit{The vertical grey lines indicate the fundamental domain. A $\Z_p$-symmetric periodic orbit needs to consist of at least $p$ magnetic geodesic arcs, this is what we refer to as $\Z_p$-symmetric configuration. Two examples are illustrated above.}}
    \label{symmetric}
\end{figure}
\noindent
\textbf{Step 3: Constructing an admissible Hamiltonian}\\
We can now reparametrize the magnetic billiard dynamics to write down a Hamiltonian with no fast periodic orbits.

\begin{Proposition}\label{lower bound period of Hamiltonian}
    We define the following Hamiltonian to reparametrize the dynamics induced by $E+V_\varepsilon$, 
    $$
    H_\varepsilon: D_{1-2\varepsilon}L(p;1)\to \R;\ (x,v)\mapsto h_{\varepsilon}\left(E(x,v)+V_\varepsilon(x) \right)
    $$
 where $h_{\varepsilon}(y)$ is given by $$
    h_{\varepsilon}:[0,\infty)\longrightarrow \R: y\mapsto 2\pi \begin{cases}
    \frac{y}{\sqrt{\varepsilon}}+\frac{\sqrt{\varepsilon}}{2}, & \textit{ for }  y\leq \frac{\varepsilon}{2}\\
          \sqrt{2y},& \textit{ for }  y \geq \frac{\varepsilon}{2}
    \end{cases}.$$
    Then periodic orbits $(\gamma, \dgc)$ with respect to $\dd\lambda-\varepsilon\pi^*\dd\alpha$ have periods
    $$
    T_{H_\varepsilon}\geq 1-\mathcal{O}(\varepsilon).
    $$
\end{Proposition}
\begin{Remark}
    Note that $h_{\varepsilon}$ is differentiable in particular at $y=\frac{\varepsilon}{2}$. The function $h_\varepsilon$ is not smooth, i.e.\ of class $C^{\infty}$, but one can approximate it by a smooth function, and as the period only sees the first derivative this will not change the bound on the period. In addition, for the reader's convenience, we add an example of a possible  graph of $h_{\varepsilon}$. 
    \begin{figure}[h]
	\centering
 \includegraphics[width=0.35\textwidth]{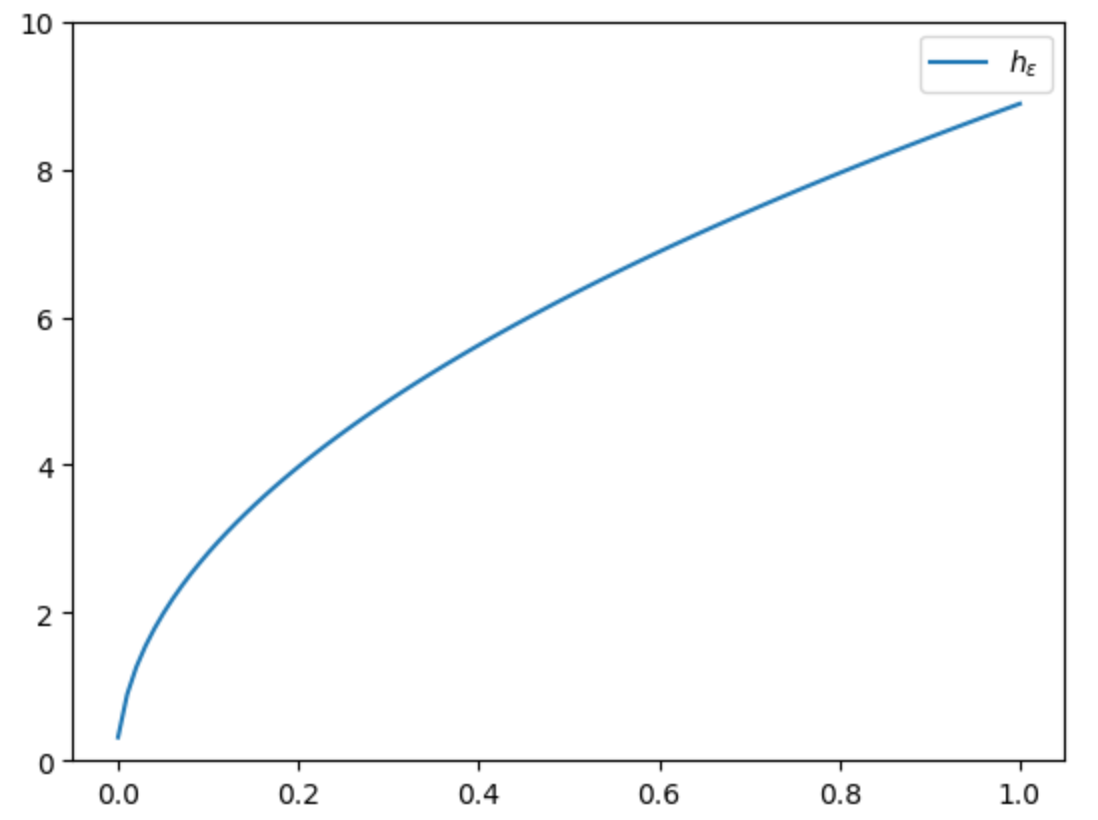}
	\caption{\textit{The figure shows a plot of the function $h_{\varepsilon}$ for $\varepsilon=\frac{1}{10}$. }}
    \label{figure plot heps}
\end{figure}
\end{Remark}
\begin{proof}
Let $(\gc, \dgc)$ be a periodic orbit of $H_{\varepsilon}$. There are two cases that we have to investigate. We start with $E(\gc,\dgc)+V_\varepsilon(\gc)\leq  \frac{\varepsilon}{2}$. Using that $H_\varepsilon =h_{\varepsilon}\circ (E(\gc,\dgc)+V_\varepsilon(\gc))$ for $h_{\varepsilon}'(y)= \frac{2\pi}{\sqrt\varepsilon}$ if $y\leq \frac{\varepsilon}{2}$, we find 
$$ 
T_{H_\varepsilon}= \frac{T_{E+V_\varepsilon}}{h_{\varepsilon}'(E(\gc,\dgc)+V_\varepsilon(\gc))}= \frac{\sqrt\varepsilon}{2\pi} T_{E+V_\varepsilon}\overset{\ref{prop potential}}{\geq} 1-\mathcal{O}(\varepsilon).
$$
Now if $E(\gc, \dgc)+V_\varepsilon(\gc)\geq  \frac{\varepsilon}{2}$, recall that $H_\varepsilon =h_\varepsilon\circ (E+V_\varepsilon)$ with $h_\varepsilon(y)=2\pi\sqrt{2y}$. Thus we find
    \[
T_{H_\varepsilon}= \frac{T_{E+V_\varepsilon}}{h'(E(\gc,\dgc)+V_\varepsilon(\gc))}=\frac{\sqrt{2(E(\gc,\dgc)+V_\varepsilon(\gc))}}{2\pi}  T_{E+V_\varepsilon}=\frac{\max\vert v\vert}{2\pi}  T_{E+V_\varepsilon}\overset{\ref{prop potential}}{\geq} 1-\mathcal{O}(\varepsilon).
    \]
\end{proof}
\noindent
We are finally in the position to prove the lower bound of the Hofer--Zehnder capacity.
\begin{Corollary}\label{lower bound Hofer Zehnder capacity}
The Hofer-Zehnder capacity $c_{\hz}(D_1L(p;1),\dd\lambda)$ is bounded from below by  
$$c_{\hz}(D_1L(p;1),\dd\lambda)\geq pl=2\pi,$$
    where $l$ denotes the length of the shortest prime geodesics on $L(p;1)$.
\end{Corollary}
\begin{proof} By monotonicity of capacities and using the shift     of the zero-section \eqref{shift of zero section} we find
    $$
    c_{HZ}(D_{1-2\varepsilon}L(p;1),\dd\lambda-\varepsilon\pi^*\dd\alpha)\leq c_{HZ}(D_1L(p;1),\dd\lambda).
    $$
    Denote the sub level set of $H_\varepsilon$ by
    $$
    U_\varepsilon=\left\lbrace (x,v)\in TL(p;1)\ \Big\vert\ \ H_\varepsilon (x,v) \leq 2\pi(1-2\varepsilon)\right\rbrace.
    $$
    Observe that $U_\varepsilon\subset D_{1-2\varepsilon}L(p;1)$. Now we can modify
    $$H_\varepsilon: U_\varepsilon\to \left [\frac{\sqrt\varepsilon}{2}, 2\pi(1-2\varepsilon)\right ]$$ 
    by composing it with another function 
    $$
    f_{\varepsilon}: \left [\frac{\sqrt\varepsilon}{2}, 2\pi(1-2\varepsilon)\right ]\to [0,\infty),
    $$
    with $0\leq f_{\varepsilon}^{\prime}<1-2\mathcal{O}(\varepsilon)$, $\min(f_{\varepsilon}\circ H_\varepsilon)=0$ and $\max(f_{\varepsilon}\circ H_\varepsilon)=\max(H_{\varepsilon})-\mathcal{O}(\varepsilon)$, so that $f_{\varepsilon}\circ H_\varepsilon$ is admissible. Indeed the periods then satisfy
    $$
    T_{f_{\varepsilon}\circ H_\varepsilon}=\frac{T_{H_\varepsilon}}{f^{\prime}_{\varepsilon}(H_\varepsilon)}>1.
    $$
    In particular, we obtain
    $$
    c_{HZ}(D_1L(p;1),\dd\lambda)\geq c_{HZ}(U_\varepsilon,\dd\lambda-\varepsilon\pi^*\dd\alpha)\geq \max(f_\varepsilon\circ H_\varepsilon)= 2\pi -\mathcal{O}(\varepsilon).
    $$
    Taking $\varepsilon\to 0$ yields the corollary as the sphere covering the lens space was of radius $1$ and therefore its prime geodesic have length $2\pi$, while the shortest prime geodesics on $L(p;1)$ have length $l=2\pi/p$.
\end{proof}

\section{Upper bound: Capacity of $(D_1S^3,\dd\lambda)$}
We will deduce the upper bound for disk tangent bundles of the lens spaces from the fact that they are all covered by the disk tangent bundle of the 3-sphere. To do so we first need to compute the Hofer--Zehnder capacity of $(D_1S^3,\dd\lambda)$.
\begin{theorem} Take the disk bundle with respect to the round metric and denote by $l$ the length of all prime geodesics, then
    $$c_{\hz}(D_1S^3,\dd\lambda)=l.$$
\end{theorem}
\begin{proof}
    We prove the Theorem by finding a lower and an upper bound that coincide. Further we normalize the metric such that $l=2\pi$, the general case follows by scaling.\\
    \ \\    
    \textit{Lower bound:} All orbits of $H: D_1S^3\setminus 0_{TS^3}\to\R,\ (x,v)\mapsto 2\pi\vert v\vert$ are periodic of period one. Composing the Hamiltonian $H$ with a suitable function $f:[0,2\pi]\to\R$ we can extend $f\circ H$ smoothly to the zero-section such that $f\circ H$ is admissible and $\max(f\circ H)=\max(H)-\mathcal{O}(\varepsilon)$. Now the lower bound follows from $\varepsilon\to 0$. 
\begin{Remark}
    For the upper bound we will use results by Hofer--Viterbo \cite{HV92} and Lu \cite{Lu06}, using the theory of pseudoholomorphic curves or more precisely Gromov--Witten invariants to obtain an upper bound for the Hofer--Zehnder capacity. For a detailed introduction to Gromov-Witten invariants in symplectic topology we refer to \cite{DS12}. We use the notation of Wendl \cite{Wdl18} and denote by $\mathrm{GW}_A([pt.],\alpha,\beta)$ the count of pseudoholomorphic spheres $u:\cp^1\to Q^3$ representing the homology class $A\in H_2(Q^3,\Z)$ with three marked points mapped to representatives of the homology classes $[pt.], \alpha,\beta \in H_2(Q^3,\Z)$. 
\end{Remark}
    \noindent
    \textit{Upper bound:}
    As shown for example in \cite[Lem.\ 3.2]{Ad22} the disk tangent bundle is symplectomorphic to the complement of 2-quadric in the 3-quadric, i.e.
    $$
    (D_1S^3,\dd\lambda)\cong (Q^3\setminus Q^2,2\sigma).
    $$
    Here $Q^n:=\lbrace [z_0:\ldots:z_{n+1}]\in \cp^{n+1}\vert \ z_0^2+\ldots +z_{n+1}^2=0\rbrace$ and $\sigma$ denotes the restriction of the Fubini-Study to $Q^3\subset \cp^4$ normalized to take value $\pi$ on the generator of $H_2(\cp^n,\Z)$.\\
    \ \\
    \noindent
    Note that $Q^3\cong \mathrm{SO}(5)/(\mathrm{SO}(3)\times\mathrm{SO}(2))$ is actually an irreducible Hermitian symmetric space of rank 2. By \cite[Lem.\ 15]{LMZ13} a Gromov--Witten invariant of the form $\mathrm{GW}_A([pt.],\alpha,\beta)$ for $A$ the generator of $H_2(Q^3,\Z)$ does not vanish. As $[Q^2]\neq 0$ by Lemma \ref{l2} we have $[Q^2]\cdot [A]\neq 0$ and therefore we find
    $$
    \mathrm{GW}_A\left([pt.],\alpha,\beta,[Q^2]\right)=\mathrm{GW}_A\left([pt.],\alpha,\beta\right)\left (A\cdot [Q^2]\right)\neq 0.
    $$
    Finally we can use a corollary \cite[Cor.\ 3.2.11]{Bim} of a theorem by Lu \cite[Thm.\ 1.10]{Lu06} to find 
    $$
    c_{\hz}(D_1S^3,\dd\lambda)=c_{\hz}(Q^3\setminus Q^2,2\sigma)\leq 2\sigma(A)=2\pi.
    $$ 
\end{proof}
\begin{Remark}
    Observe that $A$ is the generator of $H_2(Q^3,\Z)\cong\Z$, therefore it is a minimal class and the moduli space of pseudoholomorphic curves in class $A$ is compact. The Gromov-Witten invariant we use here is thus defined without using any virtual techniques.
\end{Remark}
\noindent
We can now deduce an upper bound for the Hofer--Zehnder capacity of $(D_1 L(p;1),\dd\lambda)$.
\begin{Corollary}\label{upper bound Hofer Zehnder capacity}
The Hofer-Zehnder capacity $c_{\hz}(D_1L(p;1),\dd\lambda)$ is bounded from above by
    $$c_{\hz}(D_1L(p;1),\dd\lambda)\leq pl=2\pi,$$
    where $l$ denotes the length of the shortest prime geodesics.
\end{Corollary}
\begin{proof}
    Let $H:D_1L(p;1)\to \R$ be an admissible Hamiltonian, then $H\circ \dd\mathrm{pr}: D_1S^3\to\R $ is admissible for $(DS^3,\dd\lambda)$. In particular 
    $$
    \max (H)=\max (H\circ \dd\mathrm{pr}) \leq c_{\hz}(DS^3,\dd\lambda)=pl=2\pi,
    $$
    as $pl=2\pi$ is the length of the prime geodesics on the sphere covering $L(p;1)$ isometrically.
\end{proof}
\noindent
 As the upper bound of the Hofer-Zehnder capacity $(D_1 L(p;1),\dd\lambda)$ in Corollary \ref{upper bound Hofer Zehnder capacity} coincides with the lower bound constructed in the previous section in Corollary \ref{lower bound Hofer Zehnder capacity} this finishes the proof of Theorem \ref{thm}.

\section{Further directions}
We will finish the article by discussing open problems which are in our opinion closely related to \Cref{thm}. The first question one has is probably, what about other lens spaces? First note that when considering lens spaces $L(p;1)$ for even $p$ \Cref{l1} does not hold, as every geodesic on $S^3$ covers a geodesic of $L(p;1)$ at least twice. Still, as shown by the first author \cite{B23}, for $\rp^2$ (e.g.\ $p=2$) Theorem \ref{thm} continues to hold. The lower bound in this case was obtained using geodesic billiards on one of the hemispheres of $S^3$. For general odd $p$ the authors are optimistic that one could combine the two dynamics in order to find a lower bound. As the upper bound comes from $S^3$ it actually holds for any lens space $L(p;q)$. 
 \begin{Question}
     Let $p,q\in\Z$ co-prime and odd. Can one extend the proof of the lower bound in \Cref{thm} to compute the Hofer-Zehnder capacity $D_1L(p,q)$?
 \end{Question}
 \noindent
 The guess of the authors would be yes, with the following modification one should instead of studying the magnetic flow $S^3$ equipped with it's standard contact form the magnetic flow on the ellipsoid $E\left(1,q\right)$ equipped with it's standard contact form. This can be done as follows: we equip $S^3$ with the contact form $\alpha_{q}=\frac{\mathrm{i}}{2}\left(z_1\dd \bar{z_1}-\bar{z}_1\dd z_1+ \frac{1}{q} \left(z_2\dd \bar{z}_2-\bar{z}_2\dd z_2 \right) \right)$. Note that the Reeb vector field is given by $R_{(z_1,z_2)}= 2(\mathrm{i} z_1, \mathrm{i} q z_2)$ thus the Reeb flow w.r.t. $\alpha_q$ on $S^3$ is given by $\varPhi_t(z_1,z_2)= \left(e^{2\mathrm{i}t} z_1,  e^{2 q\mathrm{i}t} z_2\right)$. If we replace $\mathrm{i}$ by $\mathrm{j}$ again this Reeb flow can used to define a $\Z_p$ action on $S^3$ as follows
 \[
 (z_1, z_2)\mapsto \left(e^{2\pi \frac{1}{p}j}z_1,e^{2\pi \frac{q}{p}j}z_2\right). 
 \]
 Dividing out this $\Z_p$ action on $S^3$ we obtain the lens space $L(p,q)=S^3/\Z_p$. 
\begin{Conjecture}
     For all $p,q\in \N$ co-prime the Hofer-Zehnder capacity of the disk bundle of $L(p,q)$ is given by
     $$
 c_{HZ}\left( D_1L(p,q), \dd\lambda\right)= 2\pi. 
 $$
\end{Conjecture} 
\noindent
Lastly we also want to say a few words about higher dimensions. Again the upper bound continues to hold as also in higher dimensions, since we have
$$
(D_1S^{2n+1},\dd\lambda)\cong (Q^{2n+1}\setminus Q^{2n}, 2\sigma).
$$
For the lower bound one could try to use the fact that the magnetic geodesic flow on $S^{2n+1}$ actually stays on totally geodesic copies of $S^3$ which are invariant under the magnetic flow, see for reference \cite{ABM}. Certainly the potential $U$ one uses can be extended to $S^{2n+1}$ such that $U$ tends to infinity in a neighborhood of the Reeb orbits $\gc_k$ where $\gc_k(t)$ is a rotation in the $k$-th coordinate by the angle $2 t$. So, we end up with the following: 
\begin{Conjecture}
    For all $p \in \N$ odd the Hofer-Zehnder capacity of the disk bundle of $L(p,1, \dots, 1)= S^{2n+1}/\Z_p$ is given by
     $$
 c_{HZ}\left( D_1L(p,1, \dots, 1), \dd\lambda\right)= 2\pi. 
 $$
\end{Conjecture}\noindent\\
\textbf{Data availability} In this research no data is processed.\\
\noindent\\ 
\textbf{Declarations}\\\\
\textbf{Conflict of interest:} On behalf of all authors, the corresponding author states that there is no conflict of interest.
\appendix
\section{Integrals of motion}\label{A}

We introduce the following coordinates on $S^3\subset \C^2$,
\begin{align*}
    z_1=e^{i\varphi_1}\sin(\theta)\ \& \ z_2=e^{i\varphi_2}\cos(\theta)
\end{align*}
for $\varphi_1,\varphi_2\in [0,2\pi)$ and $\theta\in[0,\pi/2]$. Observe that $\gamma_+$ corresponds to $\theta=\pi/2$ and $\gamma_-$ corresponds to $\theta=0$ and that
$$
\mathrm{dist}_{\partial\Omega_{\varepsilon}}(x)=\left\{\begin{array}{ll} \theta(x), & \theta(x)\leq \pi/4 \\
\pi/2-\theta(x), & \theta(x)\geq \pi/4\end{array}\right. .
$$
Hence, $V_\varepsilon$ depends only on $\theta(x)$. The metric in these coordinates is 
$$
\dd s^2=\dd\theta^2+\sin(\theta)^2\dd\varphi_1^2+\cos(\theta)^2\dd\varphi_2^2.
$$
We plug $v=\dot \theta\partial_{\theta}+\dot\varphi_1\partial_{\varphi_1}+\dot\varphi_2\partial_{\varphi_2}$ into $\dd s^2$ to obtain
$$
\vert v\vert^2=\dot\theta^2+\sin(\theta)^2\dot\varphi_1^2+\cos(\theta)^2\dot\varphi_2^2.
$$
Using the relation \( ix = \partial_{\varphi_1} + \partial_{\varphi_2} \), we can express the Lagrangian in these coordinates: 
$$
L\left( \varphi_1, \varphi_2, \theta, \dot{\varphi_1}, \dot{\varphi}_2, \dot{\theta}\right)=\frac{1}{2}\left(\dot\theta^2+\sin(\theta)^2\dot\varphi_1^2+\cos(\theta)^2\dot\varphi_2^2\right)+\varepsilon\left(\sin(\theta)^2\dot\varphi_1+\cos{\theta}^2\dot\varphi_2\right)-V_\varepsilon(\theta).
$$
In particular, the Lagrangian $L$ does not depend on $\varphi_1$ and $\varphi_2$. therefore the conjugate momenta 
\begin{align}\label{iom}
    c_1:=\frac{\partial L}{\partial\dot\varphi_1}=\sin^2(\theta)(\dot\varphi_1+\varepsilon)\ \ \&\ \ \
    c_2:=\frac{\partial L}{\partial\dot\varphi_2}=\cos^2(\theta)(\dot\varphi_2+\varepsilon)
\end{align}
are preserved. Since the energy \( E_{\varepsilon} = E - V_{\varepsilon} \), where \( E \) denotes the kinetic Hamiltonian, is also an integral of motion, the Lagrangian system \( L_{\varepsilon} \) has three independent integrals of motion.
 So the corresponding Hamiltonian system $H_{\varepsilon}$ on $TS^{3}$ is integrable, so we can conclude: 
\begin{Corollary}
    The Hamiltonian system $\left(TS^3, E+V_\varepsilon, \dd\lambda-\varepsilon\pi^*\dd\alpha\right)$ is integrable.
\end{Corollary}

\section{Bound on periods}\label{B}
Let $(\gamma_\varepsilon, \dgc_\varepsilon)$ be a periodic orbit of $E+V_\varepsilon$. We consider the following cases.
$$
\gamma_\varepsilon: \ \ \begin{cases}\text{avoids both caps}, \\
       \text{enters a cap}:\ \ \begin{cases}\text{trapped}, \\
       \text{not trapped}:\ \ \begin{cases}\text{enters only one cap}, \\
       \text{enters both caps}.
    \end{cases}
    \end{cases}
    \end{cases}
$$
As we added the potential $V_\varepsilon$, the norm of the velocity $\vert v\vert$ is not a constant of motion. However, every orbit leaves the region where the potential is non-vanishing and by conservation of energy $E+V_\varepsilon$ the velocity $\vert v\vert$ is maximal there. From now on we denote $c:=\max\vert v\vert$ the maximal velocity.\\
\ \\
\noindent
\textbf{Avoids both caps:} If $\gamma_\varepsilon$ avoids both caps it is a magnetic geodesic and therefore satisfies the bound in Lemma \ref{Lower bound for magnetic geodesics on S3}.\\
\ \\
\noindent
\textbf{Trapped:} The periodic orbit is considered 'trapped' if it stays near the cap in the following sense
$$
d(\gamma(t), \gamma_\pm) \leq \sqrt{\varepsilon} \ \ \forall t.
$$
In coordinates, this condition implies that either $\theta \geq \frac{\pi}{2} - \sqrt{\varepsilon}$ or $\theta \leq \sqrt{\varepsilon}$. Both cases work analogously, so we only consider the case $\theta \leq \sqrt{\varepsilon}$. Let $\gamma_\varepsilon: \R/T\Z \to S^3$ be a $T$-periodic trapped orbit. Recall that the potential explodes near $\gamma_- = \lbrace \theta = 0 \rbrace$, hence $\theta$ cannot vanish along $\gamma_\varepsilon$, which implies that the winding in the fibers is well-defined. More precisely, the map
$$
(e^{i\varphi_1} \sin(\theta), e^{i\varphi_2} \cos(\theta)) \mapsto e^{i\frac{\varphi_1 + \varphi_2}{2}}
$$
is well-defined on $\lbrace 0 < \theta < \sqrt{\varepsilon} \rbrace$ and is a Riemannian submersion. Recall that $\delta_\varepsilon = \langle \dot\gamma_\varepsilon, i\gamma_\varepsilon \rangle$ is, by \eqref{iom}, a conserved quantity, given by
$
\delta_{\varepsilon} = c_1 + c_2 - \varepsilon,
$
from which we can conclude that
$$
T \geq \frac{2\pi}{|\delta_\varepsilon|} \geq \frac{2\pi}{c_\varepsilon},
$$
whenever $\delta_\varepsilon \neq 0$. Now assume $\delta_\varepsilon = 0$.

\vspace{4px}

Recall from \cite[Rmk. 1.13]{ABM} that if $\gc$ is a magnetic geodesic with strength $\varepsilon$, speed $c$, and enclosing an angle $\delta$ with the Reeb vector field $R$ on $S^3$, then the projection $\pi \circ \gc$ of $\gc$ onto $S^2\left( \frac{1}{2}\right)$ for the Hopf fibration is a closed curve of radius
\[
R_{c, \delta, \varepsilon} = \sqrt{\frac{(c-\delta)(c+\delta)}{\varepsilon^2 + 4(c^2 - \varepsilon \delta)}}.
\]
This result may be surprising, as one would expect that the radius of projections of magnetic geodesics only depends on the speed and the velocity.

\vspace{4px}

On one hand, we know that $\frac{\sin(\sqrt{\varepsilon})}{2} \geq R_{c_\varepsilon, \delta_\varepsilon, \varepsilon}$, as otherwise the orbit would not be trapped. Furthermore, we assume $\delta_\varepsilon = 0$, hence by approximating $\sin(\sqrt{\varepsilon}) \approx \sqrt{\varepsilon}$, we find
$$
\frac{\varepsilon}{4} \geq R_{c_\varepsilon, 0, \varepsilon}^2 \geq \frac{c_\varepsilon^2}{\varepsilon^2 + 4 c_\varepsilon^2} \ \Leftrightarrow \varepsilon^3 \geq 4(1-\varepsilon)c_\varepsilon^2.
$$
For small values of $\varepsilon$, this leads to the following inequality
\begin{equation}\label{bound c}
\sqrt{\varepsilon} > \frac{\varepsilon^{3/2}}{2} \geq c_\varepsilon.
\end{equation}
\noindent
Thus, we need to show the following, up to terms of order $\varepsilon$:
$$
T \geq \frac{2\pi}{\sqrt{\varepsilon}}.
$$
By choosing the potential $V_\varepsilon$ to increase extremely rapidly, once $\theta < \varepsilon$, we may assume that $\min \theta \geq \varepsilon - \varepsilon^k$ for $k$ sufficiently large. Note that such a choice of the potential implies that $\gamma_\varepsilon$ does not enter the region $\lbrace \theta < \varepsilon \rbrace$ up to a negligible error, and thus has length at least $2\pi \varepsilon$. Using \eqref{bound c}, we find $c \leq \varepsilon^{3/2}$, and we obtain
$$
T > \frac{2\pi \varepsilon}{\varepsilon^{3/2}} = \frac{2\pi}{\sqrt{\varepsilon}},
$$
as wished.

\ \\
\noindent
\textbf{Enters only one cap:} If $\gamma_\varepsilon$ enters only one cap and is not trapped, we know that the magnetic geodesic it follows outside the caps is of (euclidean) radius $R>\frac{1}{2}\sin\left(\frac{\sqrt{\varepsilon}+\varepsilon}{2}\right)$. Denote $d$ the (euclidean) distance, between the center of the cap and the center of the magnetic geodesic and $r=\frac{1}{2}\sin(\varepsilon)$ the radius of the cap. Now 
$$
\cos(\alpha_\varepsilon)=\frac{d-x}{R};\ \ x=\frac{r^2+d^2-R^2}{2d},
$$
thus using the estimates $R-r<d<R+r$ we find 
$$
\cos (\alpha_\varepsilon)=\frac{d^2-r^2+R^2}{2dR}\geq\frac{R^2+(R-r)^2-r^2}{2(R+r)R}=\frac{R-r}{R+r}.
$$
We may assume $\varepsilon$ to be small so that we can approximate $\sin(\varepsilon)=\varepsilon,\ \sin\left (\frac{\sqrt\varepsilon+\varepsilon}{2}\right)=\frac{\sqrt\varepsilon+\varepsilon}{2}$ to find
$$
\cos{\alpha_\varepsilon}\geq \frac{1+\sqrt{\varepsilon}}{1+3\sqrt{\varepsilon}}=1-\mathcal{O}(\varepsilon).
$$
It follows that 
$$
\alpha_\varepsilon=\mathcal{O}(\varepsilon).
$$
This means as $\varepsilon\to 0$ these orbits either approximate periodic bounce orbits of type (2) or have length tending to infinity. In particular
$$
T_{E+V_\varepsilon}=T_E-\mathcal{O}(\varepsilon).
$$

\noindent
\textbf{Enters both caps:} Periodic orbits that enter both caps either approximate orbits of type (3) or their length tends to infinity as $\varepsilon\to 0$. In particular 
$$
T_{E+V_\varepsilon}>T_E-\mathcal{O}(\varepsilon).
$$

\section{Cohomology of the quadric}\label{C}
The cohomology of $Q^3$ is well known to experts, but for the sake of completeness, we add its computation in this article.  The main tools are the Lefschetz hyperplane theorem and the Mayer-Vietoris sequence. 
    \begin{Lemma}\label{l2}
    The cohomology of the complex quadric $Q^3$ is given by
        $$
            H^*(Q^3,\Z)=\left\{\begin{array}{ll} \Z, & *=0,2,4,6 \\
0, & *=1,3,5 \end{array}\right.
        $$
        Further the generators $A\in H_2(Q^3,\Z)$ and $B\in H_4(Q^3,\Z)$ have intersection equal to 1, i.e.\ $A\cdot B=1$.
    \end{Lemma}
    \begin{proof}
       Using the direct corollaries \cite[Cor.\ 1.24 \& 1.25]{V03} of the Lefschetz hyperplane theorem for hypersurfaces, we immediately find $H^*(\cp^3,\Z)\cong H^*(Q^3,\Z)$ for $*\neq 3$. To find $H^3(Q^3,\Z)\cong H_3(Q^3,\Z)$ one can for example apply the Mayer--Vietoris sequences to the decomposition 
       $
        Q^3\cong DS^3\cup NQ^2,
       $
       where $NQ^2$ denotes the disk normal bundle of $Q^2$. Observe that $DS^3\simeq S^3$, $NQ^2\simeq Q^2\simeq S^2\times S^2$ and $DS^3\cap NQ^2\simeq S^3\times S^2$ as the tangent bundle of $S^3$ is trivializable. therefore all terms but $H_3(Q^3,\Z)$ are known. The statement about the intersection again follows directly from \cite[Cor.\ 1.25]{V03}.
    \end{proof}

\begin{figure}[h]
	\centering
 \includegraphics[width=0.5\textwidth]{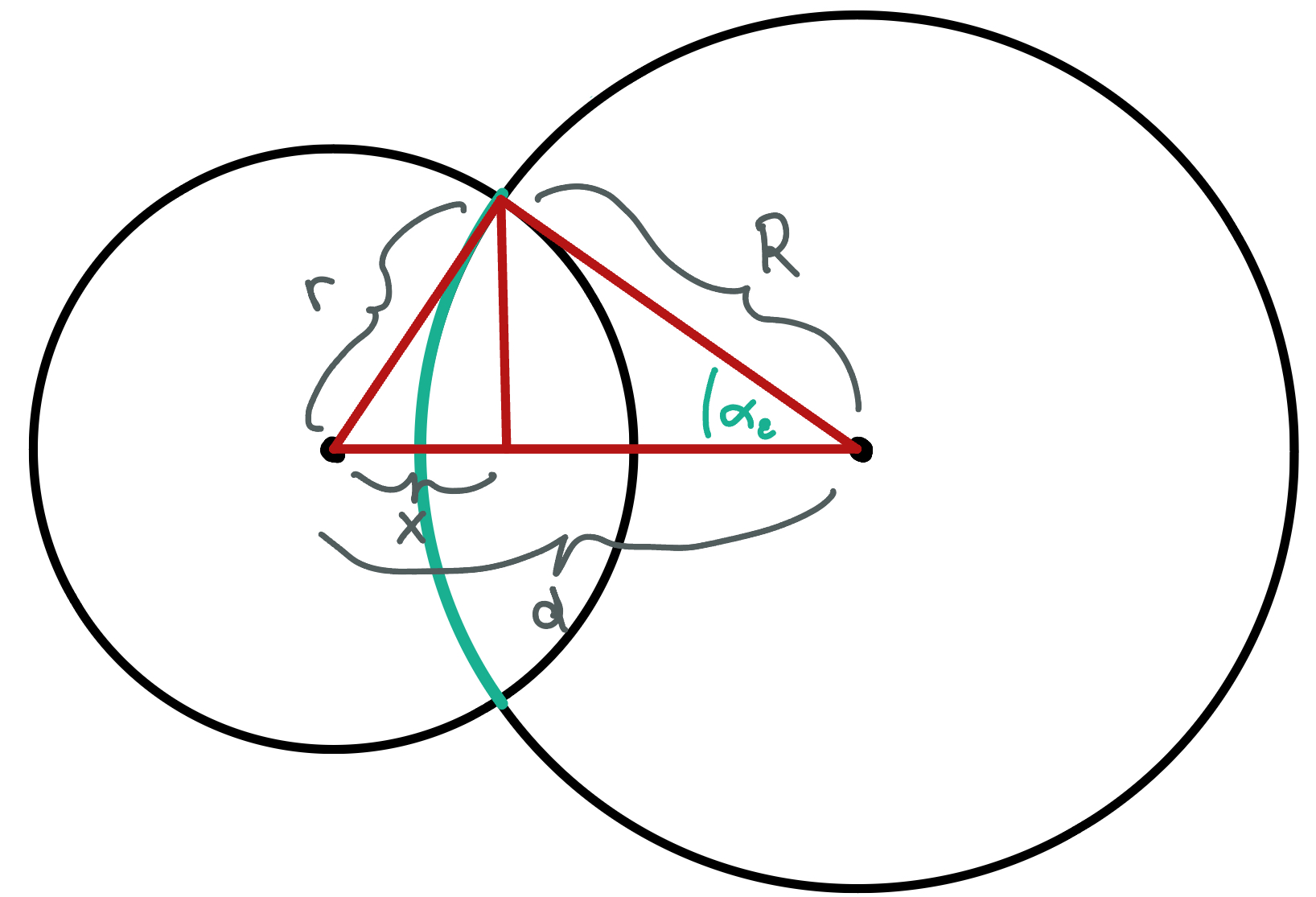}
	\caption{\textit{The circle of radius $r$ represents the cap. The circle of radius $R$ the magnetic geodesic. We need to show that the angle $\alpha_\varepsilon$ tends to zero, as this also means that most of the time the trajectory stays outside the cap where the potential does not vanish.}}
    \label{angle}
\end{figure}
\newpage
\bibliographystyle{abbrv}
\bibliography{ref}

\end{document}